\documentclass{amsart}
\usepackage{amsmath,amssymb}
\usepackage{amsthm}
\usepackage{bm}
\usepackage{graphicx}
\usepackage{ascmac}

\newtheorem{definition}{Definition}[section]
  \newtheorem{theorem}[definition]{Theorem}
   \newtheorem{lemma}[definition]{Lemma}
  \newtheorem{corollary}[definition]{Corollary}
  \newtheorem{proposition}[definition]{Proposition}
  \newtheorem{claim}[definition]{Claim}
  \newtheorem{assertion}[definition]{Assertion}

\newtheorem{remark}[definition]{Remark}

\title{Heegaard splittings of distance exactly $n$}
\author{Ayako Ido, Yeonhee Jang and Tsuyoshi Kobayashi}
\address{Department of Mathematics, Nara Women's University, 
Kitauoyanishi-machi, Nara 630-8506, Japan}

\begin{document}

\begin{abstract}
In this paper, we show that, for any integers $n\geq 2$ and $g\geq 2$, there exist genus-$g$ Heegaard splittings of compact 3-manifolds with distance exactly $n$. 
\end{abstract}


\maketitle

\section{Introduction}
For a closed orientable 3-manifold $M$, we say that $V_{1}\cup_{S} V_{2}$ is a {\it Heegaard splitting} of $M$ if $V_{1}, V_{2}$ are handlebodies such that $M=V_{1}\cup V_{2}$ and $\partial V_{1}=\partial V_{2}=S$. 
Heegaard splittings of compact orientable 3-manifolds with nonempty boundaries can be defined similarly (see Section \ref{sec-pre}).
In \cite{He}, Hempel gave the definition of \textit{distance} of Heegaard splitting by using curve complex introduced by Harvey \cite{Ha}, and showed that there exist arbitrarily high distance Heegaard splittings for closed 3-manifolds by using a construction of Kobayashi \cite{K2}. 
The manifolds are obtained by gluing two copies of handlebodies along their boundaries by the $n$th power of a pseudo-Anosov map for each $n$. 
Abrams and Schleimer \cite{AS} showed that the distance of the Heegaard splitting grows linearly with respect to $n$ by using the result of Masur and Minsky \cite{MM1}.  Moreover, Evans \cite{E} gave a combinatorial method to construct Heegaard splittings of high distance.
The main purpose of this paper is to give an answer to the following question. 

\medskip
\noindent
{\bf Question}
Given $n\geq 1$ and $g\geq 2$, does there exist a genus-$g$ Heegaard splitting with distance exactly $n$?
\medskip

For certain values, there are known examples that answer the above question affirmatively. 
For example, Berge and Scharlemann \cite{BS} showed that there exist 3-manifolds each of which admits  genus-$2$ Heegaard splittings with distance exactly $3$. 

In this paper, for each integer $n\geq 2$, we first construct a Heegaard splitting of a compact orientable 3-manifold with nonempty boundary which has distance exactly $n$.

\begin{theorem}\label{thm-1}
For any integers $n\geq 2$ and $g\geq 2$, there exists a genus-$g$ Heegaard splitting $C_{1}\cup_{P}C_{2}$ with distance exactly $n$, where $C_{1}$ and $C_{2}$ are compression bodies. 
\end{theorem}

To prove Theorem \ref{thm-1}, we give a method of constructing a pair of curves with distance exactly $n$. 
In fact, Schleimer \cite{S} gave a method of constructing a pair of curves with distance exactly four on the five-holed sphere by using {\it subsurface projection maps} defined by Masur and Minsky \cite{MM2} (for the definition, see Section \ref{sec-pre}). 
In Section \ref{sec-curves}, we mimic the idea of Schleimer to construct a pair of curves with distance exactly $n$ for any positive integer $n$.  
By using the pair of curves and the properties of a compression body obtained by adding a 1-handle to $S\times[0,1]$ where $S$ is a closed surface (for detail, see Section \ref{sec-disk-complex}), we give the proof of Theorem \ref{thm-1}. 
As a consequence of Theorem \ref{thm-1}, we have Corollary \ref{thm-2}.

\begin{corollary}\label{thm-2}
For any integers $n\geq 2$ and $g\geq 2$, there exists a genus-$g$ Heegaard splitting of a closed 3-manifold with distance exactly $n$.
\end{corollary}

\begin{remark}
{\rm 
In \cite{QZG}, the statement completely including Corollary \ref{thm-2} is given.
In fact, they show in \cite[Theorem 1]{QZG} that for each pair of integers $n\geq 1$ and $g\geq 2$ with $(n,g)\neq (1,2)$, there is a genus-$g$ Heegaard splitting of a closed 3-manifold with distance $n$.
We note that it is also remarked in \cite{QZG} that the pair $(n,g)=(1,2)$ is not realizable.
}
\end{remark}

After finishing the first version of this paper, Ruifeng Qiu informed us that there was a gap in the proof of the theorem which we used in the proof of Corollary 1.2 in the first version, and he sent us the paper \cite{MQZ}. 
We tried to fill the gap, and could give a partial result which works to give the corollary in our setting.
By the way, we learned from \cite{MQZ} that Li \cite{Li} gave a very sharp estimation of the radius of the image of the disk complex of a compact 3-manifold by subsurface projections, and this drastically improved the conclusion of Corollary 1.2 in the first version.
When the new version was almost completed, we found that \cite{QZG} was uploaded on arXiv.org, and the main result of \cite{QZG} completely covers Corollary \ref{thm-2} of this paper. 
Further, Yanqing Zou, who is one of the authors of \cite{QZG}, informed us that our arguments in the revised version mentioned above contains a gap and made some suggestions.
Thanks to the suggestion, precisely speaking with consulting \cite{QZG}, we could give the formulation of Proposition \ref{prop-appendix} in this paper.

We thank Dr. Michael Yoshizawa for many helpful discussions, particularly for teaching us the ideas of his dissertation which includes the existence of Heegaard splitting of distance $2n$ for each integer $n\geq 1$,
and we also thank Professor Yo'av Rieck for helpful information.
We would like to especially thank Professor Ruifeng Qiu and Dr. Yanqing Zou for giving us information including the preprints \cite{MQZ, QZG}.

\section{Preliminaries}\label{sec-pre}

Let $S$ be a compact connected orientable surface with genus $g$ and $p$ boundary components. 
A simple closed curve in $S$ is {\it essential} if it does not bound a disk in $S$ and is not parallel to a component of $\partial{S}$. 
An arc properly embedded in $S$ is {\it essential} if it does not co-bound a disk in $S$ together with an arc on $\partial{S}$. 
We say that $S$ is {\it sporadic} if $g=0, p\leq 4$ or $g=1, p\leq1$. 
We say that $S$ is {\it simple}, if $S$ contains no essential simple closed curves. 
We note that $S$ is simple if and only if $S$ is a 2-sphere with at most three boundary components. 
A subsurface $X$ in $S$ is {\it essential} if each component of $\partial X$ is contained in $\partial S$ or is essential in $S$.

\medskip
\noindent
{\bf Heegaard splittings}

A connected 3-manifold $V$ is a {\it compression body} if there exists a closed (possibly empty) surface $F$ and a 0-handle $B$ such that $V$ is obtained from $F\times[0, 1]\cup B$ 
by adding 1-handles to $F \times \{1\}\cup \partial B$. 
The subsurface of $\partial V$ corresponding to $F\times\{0\}$ is denoted by $\partial_{-} V$. 
Then $\partial_{+} V$ denotes the subsurface $\partial V\setminus\partial_{-} V$ of $\partial V$. 
The genus of $\partial_{+} V$ is the {\it genus} of the compression body $V$.
A compression body $V$ is called a {\it handlebody} if $\partial_- V=\emptyset$.
A disk $D$ properly embedded in $V$ is called an {\it essential disk} if $\partial D$ is an essential simple closed curve in $\partial_{+}V$.

Let $M$ be a compact orientable 3-manifold. 
We say that $C_{1}\cup_{P}C_{2}$ is a {\it genus-$g$ Heegaard splitting} of $M$ if $C_{1}$ and $C_{2}$ are compression bodies of genus $g$ in $M$ such that $C_{1}\cup C_{2}=M$ and $C_1\cap C_2=\partial_{+}C_{1}=\partial_{+}C_{2}=P$. 
Alternatively, given a Heegaard splitting $C_{1}\cup_{P}C_{2}$ of $M$, we may regard that there is a homeomorphism $f:\partial_+ C_1\rightarrow\partial_+C_2$ such that $M$ is obtained from $C_1$ and $C_2$ by identifying $\partial_+ C_1$ and $\partial_+C_2$ via $f$.
When we take this viewpoint, we will denote the Heegaard splitting by the expression $C_{1}\cup_{f}C_{2}$.

\medskip
\noindent
{\bf Curve complexes}

Except in sporadic cases, the {\it curve complex} $\mathcal{C}(S)$ is defined as follows: each vertex of $\mathcal{C}(S)$ is the isotopy class of an essential simple closed curve on $S$, and a collection of $k+1$ vertices forms a $k$-simplex of $\mathcal{C}(S)$ if they can be realized by mutually disjoint curves in $S$. 
In sporadic cases, we need to modify the definition of the curve complex slightly, as follows. 
Note that the surface $S$ is simple unless $S$ is a torus, a torus with one boundary component, or a sphere with 4 boundary components.
When $S$ is a torus or a torus with one boundary component (resp. a sphere with 4 boundary components), a collection of $k+1$ vertices forms a $k$-simplex of $\mathcal{C}(S)$ if they can be realized by curves in $S$ which mutually intersect exactly once (resp. twice). 
The {\it arc-and-curve complex} $\mathcal{AC}(S)$ is defined similarly, as follows: each vertex of $\mathcal{AC}(S)$ is the isotopy class of an essential properly embedded arc or an essential simple closed curve on $S$, and a collection of $k+1$ vertices forms a $k$-simplex of $\mathcal{AC}(S)$ if they can be realized by mutually disjoint arcs or simple closed curves in $S$. 
Throughout this paper, for a vertex $x\in\mathcal{C}(S)$ we often abuse notations and use $x$ to represent (the isotopy class of) a geometric representative of $x$. 
The symbol $\mathcal{C}^0(S)$ (resp. $\mathcal{AC}^0(S)$) denotes the 0-skeleton of $\mathcal{C}(S)$ (resp. $\mathcal{AC}(S)$).

For two vertices $x, y$ of $\mathcal{C}(S)$, we define the {\it distance} $d_{\mathcal{C}(S)}(x, y)$ between $x$ and $y$, which will be denoted by $d_{S}(x, y)$ in brief, as the minimal number of 1-simplexes of a simplicial path in $\mathcal{C}(S)$ joining $x$ and $y$. 
For two subsets $X, Y$ of $\mathcal{C}^0(S)$, 
we define ${\rm diam}_{S}(X, Y):=$ the diameter of $X\cup Y$. 
Similarly, we can define $d_{\mathcal{AC}(S)}(x, y)$ for $x,y\in\mathcal{AC}^0(S)$ and ${\rm diam}_{\mathcal{AC}(S)}(X, Y)$ for $X,Y\subset\mathcal{AC}^0(S)$.

For a sequence $a_0,a_1,\dots,a_n$ of vertices in $\mathcal{C}(S)$ with $a_i\cap a_{i+1}
=\emptyset$ $(i=0,1,\dots,n-1)$, we denote by $[a_0,a_1,\dots,a_n]$ the path in $\mathcal{C}(S)$ with vertices $a_0,a_1,\dots,a_n$ in this order.
We say that a path $[a_0,a_1,\dots,a_n]$ is a {\it geodesic} if $n=d_S(a_0,a_n)$.

Let $V$ be a compression body. 
Then the {\it disk complex} $\mathcal{D}(V)$ is the subset of $\mathcal{C}^0(\partial_{+}V)$ consisting of the vertices with representatives bounding disks of $V$. For a genus-$g(\geq 2)$ Heegaard splitting $C_{1}\cup_{P}C_{2}$, the (Hempel) {\it distance} of $C_{1}\cup_{P}C_{2}$ is defined by $d_{P}(\mathcal{D}(C_{1}),\mathcal{D}(C_{2}))=\min\{d_{P}(x, y)\mid x\in\mathcal{D}(C_{1}), y\in\mathcal{D}(C_{2})\}$.

\medskip
\noindent
{\bf Subsurface projection maps}

Let $\mathcal{P}(Y)$ denote the power set of a set $Y$. 
Suppose that $X$ is an essential subsurface of $S$. 
We call the composition $\pi_0\circ\pi_A$ of maps $\pi_A:\mathcal{C}^{0}(S)\rightarrow \mathcal{P}(\mathcal{AC}^{0}(X))$ and $\pi_0:\mathcal{P}(\mathcal{AC}^{0}(X))\rightarrow\mathcal{P}(\mathcal{C}^{0}(X))$ a {\it subsurface projection} if they satisfy the following: for a vertex $\alpha$, take a representative $\alpha$ 
so that $|\alpha\cap X|$ is minimal, where $|\cdot|$ is the number of connected components. Then 

\begin{itemize}
\item $\pi_{A}(\alpha)$ is the set of all isotopy classes of the components of $\alpha\cap X$,
\item $\pi_0(\{\alpha_1,\dots,\alpha_n\})$ is the union, for all $i=1,\dots,n$, of the set of all isotopy classes of the components of $\partial N(\alpha_{i}\cup\partial X)$ which are essential in $X$, where $N(\alpha_{i}\cup\partial X)$ is a regular neighborhood of $\alpha_i\cup\partial X$ in $X$.
\end{itemize}

We say that $\alpha$ {\it misses} $X$ (resp. $\alpha$ {\it cuts} $X$) if $\alpha\cap X=\emptyset$ (resp. $\alpha\cap X\neq\emptyset$). 

\begin{lemma}\label{subsurface distance}
Let $X$ be as above. 
Let $[\alpha_{0}, \alpha_{1}, . . . , \alpha_{n}]$ be a path in $\mathcal{C}(S)$
such that every $\alpha_{i}$ cuts $X$. 
Then ${\rm diam}_{X}(\pi_{X}(\alpha_{0}), \pi_{X}(\alpha_{n}))\leq 2n$.
\end{lemma}

\begin{proof}
Since $d_S(\alpha_i,\alpha_{i+1})=1$ and every $\alpha_i$ cuts $X$, we have $${\rm diam}_{\mathcal{AC}(X)}(\pi_A(\alpha_i), \pi_A(\alpha_{i+1}))\leq1$$ for every $i=0,1,\dots,n-1$. 
This together with \cite[Lemma 2.2]{MM2} implies 
$${\rm diam}_{X}(\pi_0(\pi_A(\alpha_i)), \pi_0(\pi_A(\alpha_{i+1})))(={\rm diam}_X(\pi_X(\alpha_i),\pi_X(\alpha_{i+1})))\leq2.$$ 
Since ${\rm diam}_{X}(\pi_{X}(\alpha_{0}), \pi_{X}(\alpha_{n}))\leq\displaystyle\sum_{i=0}^{n-1}{\rm diam}_{X}(\pi_{X}(\alpha_{i}), \pi_{X}(\alpha_{i+1}))$, this implies 
$${\rm diam}_{X}(\pi_{X}(\alpha_{0}), \pi_{X}(\alpha_{n}))\leq 2n.$$
\end{proof}

\begin{remark}\label{rmk-non-sep}
{\rm 
Let $X$ be an essential subsurface of $S$. Suppose that $X$ is disconnected, and at least two components of $X$ are non-simple.
Then for any pair of curves $\alpha,\alpha'$ on $S$ we have ${\rm diam}_X(\pi_X(\alpha),\pi_X(\alpha'))\leq 2$.
To be precise, let $X_1$ be one of the non-simple components of $X$, and $X_2$ the union of the others.
Let $a$ and $a'$ be elements of $\pi_X(\alpha)$ and $\pi_X(\alpha')$, respectively.
If both $a$ and $a'$ are contained in $X_i$ for some $i=1,2$, say $X_1$, then we can find a curve on $X_2$ that is disjoint from $a\cup a'$, which implies $d_X(a,a')\leq 2$. 
If $a\subset X_1$ and $a'\subset X_2$ (or $a\subset X_2$ and $a'\subset X_1$), we have $d_X(a,a')\leq 1$.
Thus $d_X(a,a')\leq 2$ for any pair of elements $a\in\pi_X(\alpha)$ and $a'\in\pi_X(\alpha')$, 
and hence we have ${\rm diam}_X(\pi_X(\alpha),\pi_X(\alpha'))\leq 2$.
}
\end{remark}

\section{Disk complexes}\label{sec-disk-complex}

Let $\mathcal{D}(V)\,(\subset \mathcal{C}^0(\partial_+V))$ be the disk complex of a compression body $V$.
We have a decomposition $\mathcal{D}(V)=\mathcal{D}_{\rm nonsep}(V)\sqcup \mathcal{D}_{\rm sep}(V)$, where $\mathcal{D}_{\rm nonsep}(V)$ (resp. $\mathcal{D}_{\rm sep}(V)$) denotes the subset of $\mathcal{D}(V)$ consisting of the vertices with representatives bounding non-separating (resp. separating) essential disks of $V$.
In this section, we prove the following proposition.

\begin{proposition}\label{prop-disk-complex}
Let $V$ be a compression body obtained by adding a 1-handle to $F \times [0,1]$, where $F$ is a genus-$(g-1)$ closed orientable surface ($g\geq 2$). Then we have the following.
\begin{enumerate}
\item $\mathcal{D}_{\rm nonsep}(V)$ consists of a point, say $c_0$.
\item For each element $c_{\alpha}$ of $\mathcal{D}_{\rm sep}(V)$, there is a 1-simplex in $\mathcal{C}(\partial_+V)$ joining $c_0$ and $c_{\alpha}$.
\end{enumerate}
\end{proposition} 

\begin{remark}
{\rm
In fact, we can see that $\mathcal{D}_{\rm sep}(V)$ is a countable, infinite set and that there is no 1-simplex between $c_{\alpha}$ and $c_{\alpha'}$ for each pair $c_{\alpha},c_{\alpha'}\in \mathcal{D}_{\rm sep}(V)$.
}
\end{remark}

In the remaining of this section, $V$ denotes a compression body obtained by adding a 1-handle to $F \times [0,1]$, where $F$ is a genus-$(g-1)$ closed orientable surface ($g\geq 2$). 
Then $D$ denotes the essential disk of $V$ corresponding to the co-core of the 1-handle.
Proposition \ref{prop-disk-complex} follows from Lemmas \ref{lem-nonsepa} and \ref{lem-sepa} below.

\begin{lemma}\label{lem-nonsepa}
Any non-separating disk properly embedded in $V$ is ambient isotopic to $D$. 
\end{lemma}

\begin{proof}
Let $D'$ be a non-separating disk in $V$. 
Assume that $D$ and $D'$ intersect transversely, and $|D\cap D'|$ is minimized up to ambient isotopy class of $D'$. 

Suppose $|D\cap D'|=0$, i.e., $D\cap D'=\emptyset$.
Then $D'$ is properly embedded disk in the manifold obtained from $V$ by cutting along $D$, that is, $F\times [0,1]$.
Since any disk properly embedded in $F\times [0,1]$ is boundary parallel and $D'$ is non-separating in $V$, we see that $D\cup D'$ bounds a product region, and hence $D'$ is ambient isotopic to $D$.

Suppose $|D\cap D'|>0$.
By standard innermost disk arguments, we can see that $D\cap D'$ has no loop components. 
Note that there are at least two components of $D'\setminus D$ which are outermost in $D'$. 
Take a pair of such outermost components, say $\Delta_1$ and $\Delta_2$, which are the next to each other, i.e., there is a subarc $\beta\subset \partial D'$ such that $\beta\cap \Delta_1$ is an endpoint of $\beta$ and $\beta\cap\Delta_2$ is the other endpoint of $\beta$, and $\beta$ does not intersect any other outermost disk of $D'\setminus D$.
Note that we can retrieve $F\times [0,1]$ by cutting $V$ along $D$.
Let $D^{+}, D^{-}$ be the copies of $D$ in $F \times\{1\}$, and let $\overline{\Delta}_{1}$ (resp. $\overline{\Delta}_{2}$) be the closure of the image of $\Delta_{1}$ (resp. $\Delta_{2}$) in $F\times [0,1]$. 
Note that $\overline{\Delta}_{1}$ and $\overline{\Delta}_{2}$ are disks properly embedded in $F\times [0,1]$, and $\overline{\Delta}_{i}\cap(D^+\cup D^-)$ consists of an arc properly embedded in $D^+\cup D^-$.
Let $\Gamma_i$ $(i=1,2)$ be the disk in $F\times\{1\}$ such that $\partial\Gamma_i=\partial\overline{\Delta}_{i}$.
Without loss of generality, we may suppose $\overline{\Delta}_{1}\cap(D^+\cup D^-)=\overline{\Delta}_{1}\cap D^+$.
Note that if $D^-$ is not contained in $\Gamma_1$, we can isotope $D'$ in $V$ via the product region between $\overline{\Delta}_{1}$ and $\Gamma_1$ to reduce $|D\cap D'|$, a contradiction.
Hence, $D^-$ is contained in $\Gamma_1$.
Let $\beta$ be the arc in $\partial D'$ as above.
Then $\beta\cap D$ consists of finite number of points, say $p_0,p_1,\dots,p_n$, where $\partial\beta=\{p_0,p_n\}$, $p_0\in\partial\overline{\Delta}_{1}$, $p_n\in\partial\overline{\Delta}_{2}$, and $p_0,p_1,\dots,p_n$ are arrayed on $\beta$ in this order.
Then a small neighborhood of $p_0$ in $\beta$ is contained in a small neighborhood of $D^-$ in $F\times[0,1]$.
If the other endpoint of the subarc $\overline{p_0p_1}$ of $\beta$ is contained in $\partial D^-$, then we see that the subarc $\overline{p_0p_1}$ is an inessential arc in ${\rm Cl}(F\times\{1\}\setminus (D^+\cup D^-))$.
This shows that we can reduce $|D\cap D'|$ by an isotopy on $D'$, a contradiction.
By applying the same argument successively, we see that each subarc $\overline{p_{i-1}p_i}$ $(i=1,2,\dots,n)$ joins $D^+$ and $D^-$, and particularly, a small neighborhood of $p_n$ in $\beta$ is contained in a small neighborhood of $D^+$. This shows that $\overline{\Delta}_2\cap(D^+\cup D^-)=\overline{\Delta}_2\cap D^-$.
Then we see that $D^+$ is not contained in $\Gamma_2$, hence we have a contradiction by using the argument as above.
\end{proof}

Let $D'$ be a separating essential disk properly embedded in $V$.
By an argument similar to that in the proof of Lemma \ref{lem-nonsepa}, we can see that $D'$ is ambient isotopic to a disk disjoint from $D$.
Hence, we have the following lemma.

\begin{lemma}\label{lem-sepa}
Any separating essential disk properly embedded in $V$ can be isotoped to be disjoint from the non-separating disk $D$. 
\end{lemma}

\section{A pair of curves with distance exactly $n$}\label{sec-curves}

In this section, for each integer $n\geq 3$, we construct pairs of curves with distance exactly $n$. 
Let $S$ be a closed connected orientable surface with genus greater than or equal to $2$.
We first prove Propositions \ref{prop1} and \ref{prop2}.
Then we describe the constructions of paths in $\mathcal{C}(S)$ of length $n$ and show that they are geodesics in $\mathcal{C}(S)$.

\begin{proposition}\label{prop1}
For an integer $n(\geq 4)$,
let $[\alpha_0, \alpha_1,\dots,\alpha_n]$ be a path in $\mathcal{C}(S)$ satisfying the following.
\begin{itemize}
\item[(H1)] $[\alpha_0,\dots,\alpha_{n-2}]$ and $[\alpha_{n-2},\alpha_{n-1},\alpha_n]$ are geodesics in $\mathcal{C}(S)$,
\item[(H2)] ${\rm diam}_{X_{n-2}}(\pi_{X_{n-2}}(\alpha_{n-4}),\pi_{X_{n-2}}(\alpha_{n}))\geq 4n$, where $X_{n-2}={\rm Cl}(S\setminus N(\alpha_{n-2}))$.
\end{itemize}
Then $[\alpha_0, \alpha_1,\dots,\alpha_n]$ is a geodesic in $\mathcal{C}(S)$.
\end{proposition}

\begin{remark}\label{rem-non-sep-1}
{\rm
In Proposition \ref{prop1}, we note that $X_{n-2}$ is connected, i.e., $\alpha_{n-2}$ is non-separating in $S$. This can be shown by using Remark \ref{rmk-non-sep} together with the condition (H2).
}
\end{remark}

\begin{proof}[Proof of Proposition \ref{prop1}] 
Let $[\beta_0,\beta_1,\dots,\beta_m]$ be a geodesic in $\mathcal{C}(S)$ such that $\beta_{0}=\alpha_0$, $\beta_{m}=\alpha_n$. 
Then note that $m\leq n$.  

\begin{claim}\label{claim1}
$\beta_{j}=\alpha_{n-2}$ for some $j\in\{0,1,\dots,m\}$.
\end{claim}

\begin{proof}
Assume on the contrary that $\beta_{j}\neq \alpha_{n-2}$ for any $j$.
Then, by Remark \ref{rem-non-sep-1}, every $\beta_{j}$ cuts $X_{n-2}$. 
By Lemma \ref{subsurface distance}, we have ${\rm diam}_{X_{n-2}}(\pi_{X_{n-2}}(\beta_{0}), \pi_{X_{n-2}}(\beta_{m}))\leq2m$. 
On the other hand, since $[\alpha_0, \alpha_1,\dots,\alpha_{n-2}]$ is a geodesic, no $\alpha_i$ $(0\leq i\leq n-3)$ is isotopic to $\alpha_{n-2}$. 
Hence each $\alpha_i$ $(0\leq i\leq n-3)$ cuts $X_{n-2}$. 
By Lemma \ref{subsurface distance}, we have ${\rm diam}_{X_{n-2}}(\pi_{X_{n-2}}(\alpha_{0}), \pi_{X_{n-2}}(\alpha_{n-4}))\leq 2(n-4)<2n.$
These imply
\begin{align*}
  {\rm diam}_{X_{n-2}}(\pi_{X_{n-2}}(\alpha_{n-4}), \pi_{X_{n-2}}(\alpha_{n})) &\leq {\rm diam}_{X_{n-2}}(\pi_{X_{n-2}}(\alpha_{n-4}), \pi_{X_{n-2}}(\alpha_{0}))\\&\hspace{5mm}+{\rm diam}_{X_{n-2}}(\pi_{X_{n-2}}(\alpha_{0}), \pi_{X_{n-2}}(\alpha_{n}))\\
  &< 2n+{\rm diam}_{X_{n-2}}(\pi_{X_{n-2}}(\beta_{0}), \pi_{X_{n-2}}(\beta_{m}))\\
  &\leq 2n+2m\\  &\leq 4n.
\end{align*}
This contradicts the hypothesis (H2).
\end{proof}

By Claim \ref{claim1} and the hypothesis (H1), we have the equalities
\begin{eqnarray*}
  j=d_{S}(\beta_{0}, \beta_{j})&=&d_S(\alpha_0,\alpha_{n-2})=n-2,\\
  m-j=d_{S}(\beta_{j}, \beta_{m})&=&d_{S}(\alpha_{n-2}, \alpha_{n})=2.
\end{eqnarray*}
By combining the above equalities, we have $m=n$. 
Recall that $[\beta_0, \beta_1,\dots,\beta_m]$ is a geodesic in $\mathcal{C}(S)$ with $\beta_0=\alpha_0$ and $\beta_m=\alpha_n$.
Hence, $[\alpha_0, \alpha_1,\dots,\alpha_n]$ is a geodesic in $\mathcal{C}(S)$.
\end{proof}

\begin{proposition}\label{prop2}
For an integer $n(\geq 3)$,
let $[\alpha_0, \alpha_1,\dots,\alpha_n]$ be a path in $\mathcal{C}(S)$ satisfying the following.
\begin{itemize}
\item[(H1')] $[\alpha_0,\dots,\alpha_{n-1}]$ and $[\alpha_{n-2},\alpha_{n-1},\alpha_n]$ are geodesics in $\mathcal{C}(S)$,
\item[(H2')] $\alpha_{n-2}\cup\alpha_{n-1}$ is non-separating in $S$, and ${\rm diam}_{S'}(\pi_{S'}(\alpha_{0}),\pi_{S'}(\alpha_{n}))>2n$, where $S'={\rm Cl}(S\setminus N(\alpha_{n-2}\cup\alpha_{n-1}))$.
\end{itemize}
Then $[\alpha_0, \alpha_1,\dots,\alpha_n]$ is a geodesic in $\mathcal{C}(S)$.
\end{proposition}

\begin{proof}
Let $[\beta_0,\beta_1,\dots,\beta_m]$ be a geodesic in $\mathcal{C}(S)$ such that $\beta_{0}=\alpha_{0}$, $\beta_{m}=\alpha_{n}$. 
Then note that $m\leq n$.   

\begin{claim}
There exists $j\in\{0,1,\dots,m\}$ such that $\beta_{j}=\alpha_{n-2}$ or $\beta_{j}=\alpha_{n-1}$. 
\end{claim}

\begin{proof}
Suppose that $\beta_{j}\neq\alpha_{n-2}$ and $\beta_{j}\neq\alpha_{n-1}$ for any $j$. 
Since $\alpha_{n-2}\cup\alpha_{n-1}$ is non-separating in $S$, each $\beta_{j}$ cuts $S'$. 
Hence, by Lemma \ref{subsurface distance}, we have $${\rm diam}_{S'}(\pi_{S'}(\beta_{0}), \pi_{S'}(\beta_{m}))\leq 2m \leq 2n.$$
On the other hand, by (H2'), ${\rm diam}_{S'}(\pi_{S'}(\beta_{0}), \pi_{S'}(\beta_{m}))>2n$, a contradiction.  
\end{proof}

Suppose $\beta_{j}=\alpha_{n-2}$. Then we have the equalities
\begin{eqnarray*}
   j=d_{S}(\beta_{0}, \beta_{j})&=&d_S(\alpha_0,\alpha_{n-2})=n-2,\\
   m-j=d_{S}(\beta_{j}, \beta_{m})&=&d_{S}(\alpha_{n-2}, \alpha_{n})=2.
\end{eqnarray*}
By combining the above equalities, we have $n=m$.
Hence, $[\alpha_0, \alpha_1,\dots,\alpha_n]$ is a geodesic in $\mathcal{C}(S)$.
We can use a similar argument for the case when $\beta_{j}=\alpha_{n-1}$. 
This completes the proof of Proposition \ref{prop2}. 
\end{proof}

\subsection{A construction of a concrete example: the case when $n$ is even}\label{even}

We first assume that $n$ is an even integer with $n\geq 4$.
Let $\alpha_0$, $\alpha_2$ be essential non-separating simple closed curves on $S$ which intersect transversely in one point, and let $\alpha_1$ be an essential simple closed curve on $S$ which is disjoint from $\alpha_0\cup\alpha_2$. 
Let $X_{2}={\rm Cl}(S\setminus N(\alpha_2))$. 
Note that $[\alpha_{0}, \alpha_{1}, \alpha_{2}]$ is a geodesic of length two in $\mathcal{C}(S)$.
Choose a homeomorphism $\varphi_2:S\rightarrow S$ such that $\varphi_2(N(\alpha_2))=N(\alpha_2)$ and that ${\rm diam}_{X_{2}}(\pi_{X_{2}}(\alpha_{0}), \pi_{X_{2}}(\varphi_2(\alpha_{0})))\geq 4n$. 
This is possible by \cite[Proposition 4.6]{MM1}.
Let $\alpha_{3}=\varphi_2(\alpha_{1})$ and $\alpha_{4} =\varphi_2(\alpha_{0})$. 
Note that $[\alpha_{2}, \alpha_{3}, \alpha_{4}]$ is a geodesic of length two in $\mathcal{C}(S)$, and $|\alpha_2\cap\alpha_4|=1$. 

\begin{figure}[htbp]
 \begin{center}
 \includegraphics[width=55mm]{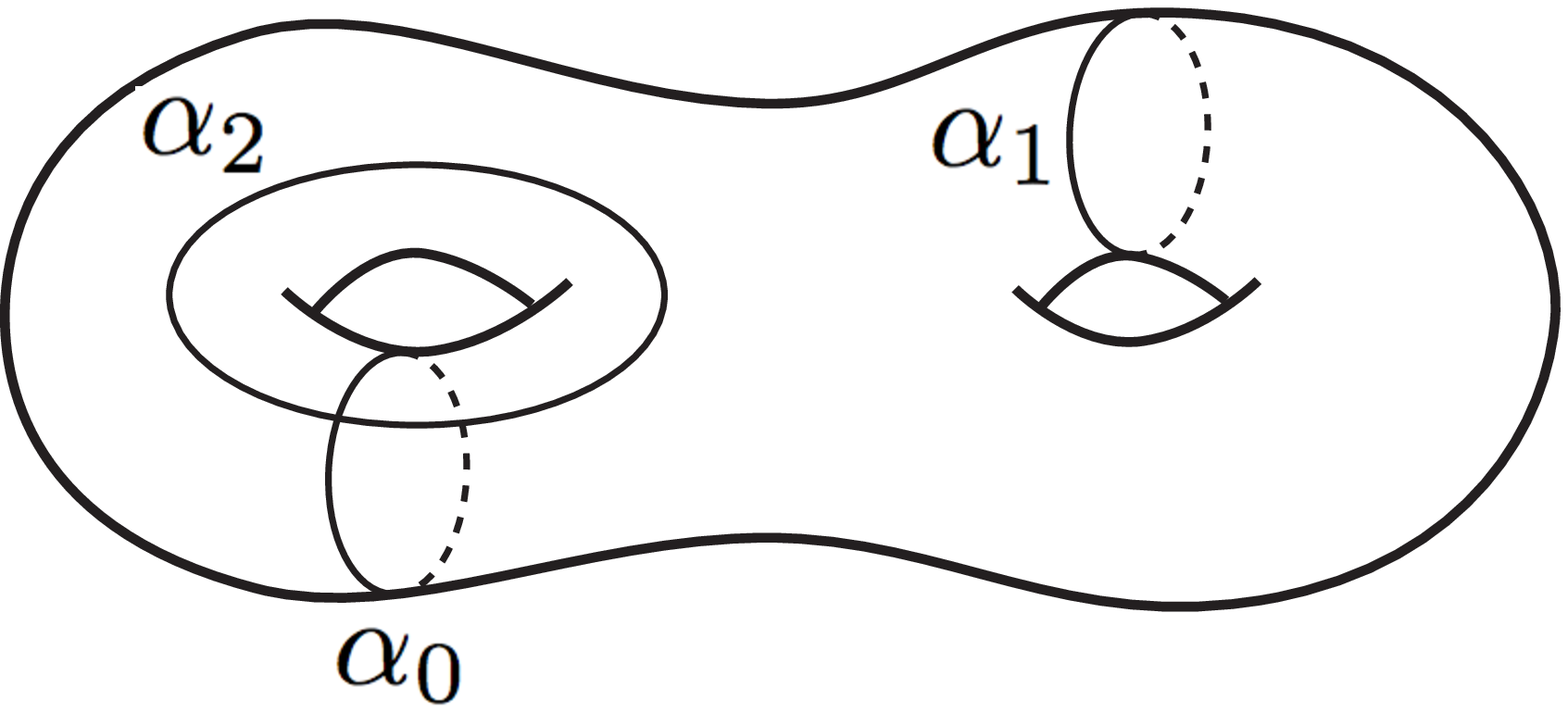}
 \end{center}
 \caption{}
\label{fig:1}
\end{figure}

We repeat this process to construct a path $[\alpha_0,\alpha_1,\dots,\alpha_n]$ inductively as follows.
Suppose we have constructed a path $[\alpha_0,\alpha_1,\dots,\alpha_i]$ with $|\alpha_{i-2}\cap\alpha_i|=1$ for an even integer $i(<n)$.
Let $X_{i}={\rm Cl}(S\setminus N(\alpha_i))$.
Choose a homeomorphism $\varphi_i:S\rightarrow S$ such that $\varphi_i(N(\alpha_i))=N(\alpha_i)$ and that 
\begin{equation}\label{eqn-f_i}
{\rm diam}_{X_{i}}(\pi_{X_{i}}(\alpha_{i-2}), \pi_{X_{i}}(\varphi_i(\alpha_{i-2})))\geq 4n.
\end{equation}
Then we let $\alpha_{i+1}= \varphi_i(\alpha_{i-1})$ and $\alpha_{i+2} = \varphi_i(\alpha_{i-2})$. 
Note that $[\alpha_{i}, \alpha_{i+1}, \alpha_{i+2}]$ is a geodesic of length two in $\mathcal{C}(S)$, and we have obtained the path $[\alpha_0,\alpha_1,\dots,\alpha_i,\alpha_{i+1},\alpha_{i+2}]$ with $|\alpha_{i}\cap\alpha_{i+2}|=1$.

\begin{assertion}\label{prop-even}
For each $k\in\{2, 4,\dots, n\}$, the path $[\alpha_0,\alpha_1,\dots,\alpha_k]$ in $\mathcal{C}(S)$ is a geodesic.
\end{assertion}

\begin{proof}
We prove the proposition by mathematical induction on $k$.
It is clear that $[\alpha_{0}, \alpha_{1}, \alpha_{2}]$ is a geodesic in $\mathcal{C}(S)$.
Hence, Assertion \ref{prop-even} holds for $k=2$.
Assume that $[\alpha_0,\alpha_1,\dots,\alpha_k]$ is a geodesic in $\mathcal{C}(S)$ for some $k\in\{2,4,\dots,n-2\}$.
We note that $[\alpha_{k}, \alpha_{k+1}, \alpha_{k+2}]$ is a geodesic in $\mathcal{C}(S)$.
Furthermore, by the inequality (\ref{eqn-f_i}), we have ${\rm diam}_{X_{k}}(\pi_{X_{k}}(\alpha_{k-2}), \pi_{X_{k}}(\alpha_{k+2}))\geq 4n\geq 4(k+2)$.
Hence, by Proposition \ref{prop1}, the path $[\alpha_0,\alpha_1,\dots,\alpha_{k+2}]$ is a geodesic in $\mathcal{C}(S)$,
which shows that Assertion \ref{prop-even} holds for $k+2$.
This completes the proof of Assertion \ref{prop-even}.
\end{proof}

\subsection{A construction of a concrete example: the case when $n$ is odd}\label{odd}

Suppose that $n$ is an odd integer with $n\geq 3$. 
Let $[\alpha_0,\alpha_1,\dots,\alpha_{n-1}]$ be a geodesic in $\mathcal{C}(S)$ as in the previous subsection. 
Here, in addition, we assume that each $\alpha_{i}$ is a non-separating curve. 
(It is easy to see that this holds if we take a non-separating curve in $S$ for $\alpha_1$ at the beginning of the construction of the geodesic.)
Note that $\alpha_{n-3}$ intersects $\alpha_{n-1}$ transversely in one point and is disjoint from $\alpha_{n-2}$.
Note also that $\alpha_{n-2}$ is non-separating.
It is easy to see that these imply that $\alpha_{n-1}\cup\alpha_{n-2}$ is non-separating.
Choose a non-separating essential simple closed curve $\gamma$ on $S$ such that $\gamma\cap\alpha_{n-1}=\emptyset$ and $\gamma$ intersects $\alpha_{n-2}$ transversely in one point. 
Let $S'={\rm Cl}(S\setminus N(\alpha_{n-2}\cup\alpha_{n-1}))$.
By \cite[Proposition 4.6]{MM1}, there exists a homeomorphism $\varphi:S\rightarrow S$ such that $\varphi(N(\alpha_{n-2}))=N(\alpha_{n-2})$, $\varphi(N(\alpha_{n-1}))=N(\alpha_{n-1})$ and ${\rm diam}_{S'}(\pi_{S'}(\alpha_{0}), \pi_{S'}(\varphi(\gamma)))>2n$. 
Let $\alpha_{n}=\varphi(\gamma)$. 
Note that $\alpha_{n}\cap\alpha_{n-1}=\emptyset$ and 
$\alpha_{n}$ intersects $\alpha_{n-2}$ transversely in one point. 
This fact implies that $[\alpha_{n-2},\alpha_{n-1},\alpha_n]$ is a geodesic in $\mathcal{C}(S)$.
On the other hand, $[\alpha_{0},\dots,\alpha_{n-1}]$ is also a geodesic in $\mathcal{C}(S)$.
Hence, by Proposition \ref{prop2} together with the above inequality ${\rm diam}_{S'}(\pi_{S'}(\alpha_{0}), \pi_{S'}(\varphi(\gamma)))>2n$, we see that the path $[\alpha_0,\alpha_1,\dots,\alpha_n]$ is a geodesic in $\mathcal{C}(S)$.

\begin{figure}[htbp]
 \begin{center}
 \includegraphics[width=55mm]{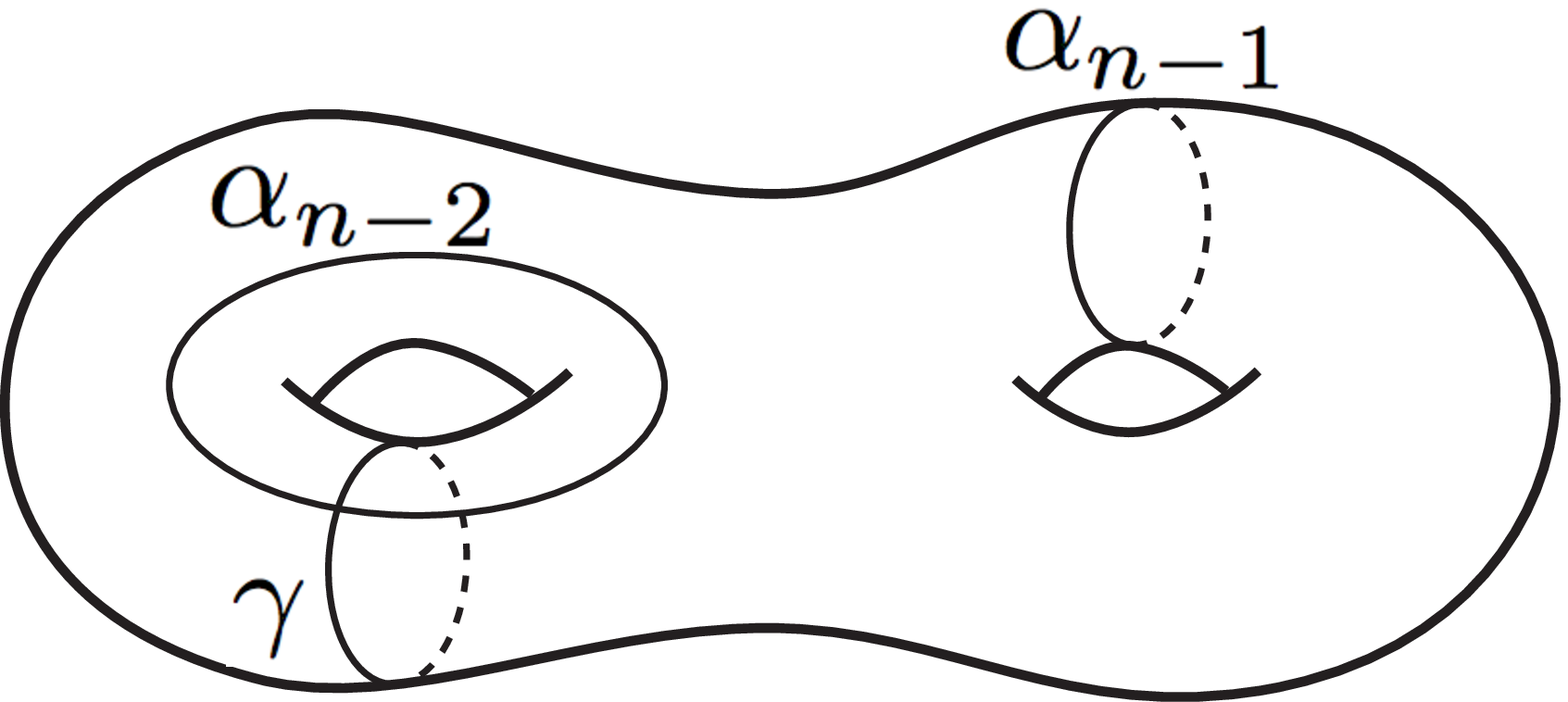}
 \end{center}
 \caption{}
\label{fig:2}
\end{figure}

\section{Proofs of Theorem \ref{thm-1} and Corollary \ref{thm-2}}

\begin{proof}[Proof of Theorem \ref{thm-1}]
Let $C_1$ and $C_2$ be copies of the compression body obtained by adding a 1-handle to $F\times [0,1]$, where $F$ is a genus-$(g-1)$ closed orientable surface ($g\geq 2$).
Let $\alpha_0$ be the boundary of the non-separating essential disk $D_1$ properly embedded in $C_1$ and $\alpha_2$ a simple closed curve on $\partial_{+} C_1$ which intersects $\alpha_0$ transversely in one point.
Then we construct a geodesic $[\alpha_0,\alpha_1,\dots,\alpha_{n+2}]$ on $\partial_{+} C_1$ as in Section \ref{sec-curves}. 
Note that $\alpha_{n+2}$ intersects $\alpha_{n}$ transversely in one point by the construction. 
Take any homeomorphism $f:\partial_+C_1\rightarrow \partial_+C_2$ such that $f(\alpha_{n+2})=\partial D_2$,
where $D_2$ is the non-separating essential disk properly embedded in $C_2$.
We identify the boundary components $\partial_+C_1$ and $\partial_+C_2$ by $f$, and let $P=\partial_+C_1=f^{-1}(\partial_+C_2)$.
Then $C_1\cup_P C_2$ is a genus-$g$ Heegaard splitting of a compact orientable 3-manifold.

Let $D_1'$ be a separating essential disk in $C_{1}$ disjoint from $\alpha_{2}$ obtained as follows. 
Let $D_1^+$ and $D_1^-$ be the components of ${\rm Cl}(\partial N(D_1)\setminus \partial_+C_1)$, where $N(D_1)$ is a regular neighborhood of $D_1$ in $C_1$. 
Take the subarc of $\alpha_2$ lying outside of the product region $N(D_1)$ betwteen $D_1^+$ and $D_1^-$.
Then $D_1'$ is obtained from $D_1^+\cup D_1^-$ by adding a band along the subarc of $\alpha_2$.
Similarly, we can obtain a separating essential disk $D_{2}'$ in $C_{2}$ disjoint from $\alpha_{n}$, by using $D_2$ and $\alpha_n$. 
On the other hand, we have $d_{P}(\alpha_{2}, \alpha_{n})=n-2$ since $[\alpha_0,\alpha_1,\dots,\alpha_{n+2}]$ is a geodesic in $\mathcal{C}(P)$.
Hence,
\begin{align*}
   d_{P}(\partial D_{1}', \partial D_{2}') &\leq d_{P}(\partial D_{1}', \alpha_{2})+d_{P}(\alpha_{2}, \alpha_{n})+d_{P}(\alpha_{n}, \partial D_{2}')\\
        &= 1+(n-2)+1\\
        &=n.
\end{align*}

Let $D_1''\subset C_1$ and $D_2''\subset C_2$ be  any essential disks.
By Proposition \ref{prop-disk-complex}, we have $d_P(\partial D_i'',\partial D_i)\leq 1$ for $i=1,2$.
This implies 
\begin{align*}
d_P(\partial D_1,\partial D_2)&\leq d_P(\partial D_1,\partial D_1'')+d_P(\partial D_1'',\partial D_2'')+d_P(\partial D_2'',\partial D_2)\\
&\leq 1+d_P(\partial D_1'',\partial D_2'')+1,
\end{align*}
and hence 
\begin{align*}
d_P(\partial D_1'',\partial D_2'')&\geq d_{P}(\partial D_{1}, \partial D_{2})-2\\
&=d_{P}(\alpha_{0}, \alpha_{n+2})-2\\
&=(n+2)-2\\&=n.
\end{align*}
Hence $d_P(\partial D_1'',\partial D_2'')\geq n$ for any pair of essential disks $D_1''\subset C_1$ and $D_2''\subset C_2$, which implies $d_{P}(\mathcal{D}(C_{1}), \mathcal{D}(C_{2}))\geq n$.
Since $d_P(\partial D_1',\partial D_2')\leq n$, we have $d_{P}(\mathcal{D}(C_{1}), \mathcal{D}(C_{2}))=n$. 
\end{proof}

In the remaining of the paper, we prove Corollary \ref{thm-2} by using the following proposition.
(Throughout this paper, given an embedding $\varphi:X\rightarrow Y$ between compact surfaces $X$ and $Y$, we abuse notations and use $\varphi$ to denote the map $\mathcal{C}^0(X)\rightarrow \mathcal{C}^0(Y)$ or $ \mathcal{P}(\mathcal{C}^0(X))\rightarrow  \mathcal{P}(\mathcal{C}^0(Y))$ induced by $\varphi:X\rightarrow Y$.)

\begin{proposition}\label{prop-appendix}
Let $V_1\cup_f V_2$ be a genus-$g(\geq 2)$ Heegaard splitting, where $V_1$ and $V_2$ are handlebodies.
Let $D_0$ be a separating essential disk in $V_1$, and let $\mathcal{D}_2$ be either $\mathcal{D}(V_2)$ or a finite subset of $\mathcal{D}(V_2)$. 
Assume that $d_{\partial V_2}(f(\partial D_0),\mathcal{D}_2)=n\geq 3$.
Then there exists a homeomorphism $g:\partial V_1\rightarrow \partial V_1$ 
such that $d_{\partial V_2}(fg(\mathcal{D}(V_1)),\mathcal{D}_2)=n$.
\end{proposition}

\begin{proof}
Let $V_1^1$ and $V_1^2$ be the closures of the two components of $V_1\setminus D_0$.
For $i=1,2$, let $F_i$ be the subsurface $\partial V_1^i\cap \partial V_1$ of $\partial V_1$,
and let $\pi_{F_i}=\pi_0\circ\pi_A^i:\mathcal{C}^{0}(\partial V_1)\rightarrow \mathcal{P}(\mathcal{AC}^{0}(F_i))\rightarrow \mathcal{P}(\mathcal{C}^{0}(F_i))$ be the subsurface projection introduced in Section 2.
Let $P_i:F_i\rightarrow F_i\cup D_0$ be the inclusion map.
Since $D_0$ is separating, the image of any essential simple closed curve in $F_i$ by $P_i$ is essential in $F_i\cup D_0$. 
This immediately implies:
\begin{claim}\label{claim-Pi}
For any non-empty subset $E$ of $\mathcal{C}^0(F_i)$, we have
\begin{itemize}
\item $P_i(E)$ is non-empty, and
\item ${\rm diam}_{F_i\cup D_0}(P_i(E))\leq{\rm diam}_{F_i}(E)$.
\end{itemize}
\end{claim}

We note that there exists a constant $N$ such that 
\begin{equation}\label{eqn-0}
{\rm diam}_{F_i}(\pi_{F_i}f^{-1}(\mathcal{D}_2))\leq N\ (i=1,2).
\end{equation}
In fact, if $\mathcal{D}_2$ is a finite subset of $\mathcal{D}(V_2)$, this is clear. 
If $\mathcal{D}_2=\mathcal{D}(V_2)$, then, by \cite[Theorem 1]{Li} together with the assumption $d_{\partial V_2}(f(\partial D_0), \mathcal{D}_2)\geq 3$, we see that ${\rm diam}_{f(F_i)}(\pi_{f(F_i)}(\mathcal{D}_2))\leq 12$, which means ${\rm diam}_{F_i}(\pi_{F_i}f^{-1}(\mathcal{D}_2))\leq 12$.
By Claim \ref{claim-Pi}, the inequality (\ref{eqn-0}) implies  
\begin{equation}\label{eqn-0-2}
{\rm diam}_{F_i\cup D_0}(P_i\pi_{F_i}f^{-1}(\mathcal{D}_2))\leq  N\ (i=1,2).
\end{equation}

Let $\mathcal{D}'(V_1^i)$ be the subset of $\mathcal{C}^0(F_i)$ consisting of simple closed curves that bound disks in $V_1^i$ ($i=1,2$).
By the inequality (\ref{eqn-0-2}) and \cite{He} (see also \cite{AS}), we see that there exists a homeomorphism $g:\partial V_1\rightarrow \partial V_1$ such that $g(\partial D_0)=\partial D_0$ and
\begin{equation}\label{eqn-1-1}
d_{F_i\cup D_0}(P_i(\mathcal{D}'(V_1^i)), \hat{g_i}^{-1}(P_i\pi_{F_i}f^{-1}(\mathcal{D}_2)))\geq 2n
\end{equation}
for each $i=1,2$, where $\hat{g_i}:F_i\cup D_0\rightarrow F_i\cup D_0$ is a homeomorphism obtained by extending $g|_{F_i}:F_i\rightarrow F_i$. 
(We note that $g|_{F_i}:F_i\rightarrow F_i$ extends to a homeomorphism $\hat{g_i}:F_i\cup D_0\rightarrow F_i\cup D_0$ in a unique way up to isotopy in $D_0$ by Alexander's trick.)
Since $g(\partial D_0)=\partial D_0$, it is easy to see that $\hat{g_i}^{-1}(P_i\pi_{F_i}f^{-1}(\mathcal{D}_2))=P_i(g|_{F_i})^{-1}\pi_{F_i}f^{-1}(\mathcal{D}_2)=P_i\pi_{F_i}g^{-1}f^{-1}(\mathcal{D}_2)=P_i\pi_{F_i}(fg)^{-1}(\mathcal{D}_2)$.
We denote the map
$P_i\pi_{F_i}(fg)^{-1}(=\hat{g_i}^{-1}P_i\pi_{F_i}f^{-1}):\mathcal{C}^{0}(\partial V_2)\rightarrow \mathcal{P}(\mathcal{C}^0(F_i\cup D_0))$ by $\Phi_i$.
Then, by the inequality (\ref{eqn-1-1}), we have 
\begin{equation}\label{eqn-1-2}
d_{F_i\cup D_0}(P_i(\mathcal{D}'(V_1^i)), \Phi_i(\mathcal{D}_2))\geq 2n\ (i=1,2).
\end{equation}

Note that $d_{\partial V_2}(fg(\mathcal{D}(V_1)), \mathcal{D}_2)\leq n$ since $f(\partial D_0)=fg(\partial D_0)\in fg(\mathcal{D}(V_1))$ and $d_{\partial V_2}(f(\partial D_0),\mathcal{D}_2)=n$ by the assumption.
To prove $d_{\partial V_2}(fg(\mathcal{D}(V_1)), \mathcal{D}_2)=n$,
assume on the contrary that $d_{\partial V_2}(fg(\mathcal{D}(V_1)), \mathcal{D}_2)<n$, or equivalently, $d_{\partial V_1}(\mathcal{D}(V_1),$ $(fg)^{-1}(\mathcal{D}_2))<n$. 
Then there exist $a\in \mathcal{D}(V_1)$ and $b\in \mathcal{D}_2$ such that 
\begin{equation}\label{eqn-a-b}
d_{\partial V_1}(a,(fg)^{-1}(b))=m<n.
\end{equation}
Let $[\gamma_0,\gamma_1,\dots,\gamma_m]$ be a geodesic in $\mathcal{C}(\partial V_1)$ from $a$ to $(fg)^{-1}(b)$.
\begin{claim}\label{claim-gamma-cuts-fi}
Every $\gamma_j$ $(j=1,2,\dots,m)$ cuts both $F_1$ and $F_2$.
\end{claim}
\begin{proof}
Assume that $\gamma_j$ does not cut $F_i$ for some $j\in\{1,2,\dots,m\}$ and some $i\in\{1,2\}$.
Then $\gamma_j$ is disjoint from $\partial D_0(=\partial F_1=\partial F_2)$, and hence we have
\begin{eqnarray*}
n&=&d_{\partial V_1}(\partial D_0,(fg)^{-1}(\mathcal{D}_2))\\
&\leq& d_{\partial V_1}(\partial D_0,\gamma_j)+d_{\partial V_1}(\gamma_j,(fg)^{-1}(\mathcal{D}_2))\\
&\leq& d_{\partial V_1}(\partial D_0,\gamma_j)+d_{\partial V_1}(\gamma_j,\gamma_m)\\
&\leq& 1+(m-j)\\
&<& 1+n-j,
\end{eqnarray*}
a contradiction.
\end{proof}

Let $D_a$ be a disk in $V_1$ bounded by $a$.
We may assume that $|D_a\cap D_0|$ is minimal.
By using innermost disk arguments, we see that $D_a\cap D_0$ has no loop components.

\vspace{2mm}
{\it Case 1}. $|D_a\cap D_0|\neq 0$.
\vspace{2mm}

Let $\Delta$ be a component of $D_a\setminus D_0$ that is outermost in $D_a$.
Then $\Delta\subset V_1^i$ for some $i=1,2$.
Without loss of generality, we may assume that $\Delta\subset V_1^1$, which implies $a(=\gamma_0)$ cuts $F_1$.
This, together with Claim \ref{claim-gamma-cuts-fi}, shows that every $\gamma_j$ $(j=0,1,\dots,m)$ in the geodesic $[\gamma_0,\gamma_1,\dots,\gamma_m]$ from $a$ to $(fg)^{-1}(b)$ cuts $F_1$.
Hence, by Lemma \ref{subsurface distance}, we have 
\begin{equation}\label{eqn-2-1}
{\rm diam}_{F_1}(\pi_{F_1}(a), \pi_{F_1}(fg)^{-1}(b))\leq 2m<2n,
\end{equation}
which implies, by Claim \ref{claim-Pi},
\begin{equation}\label{eqn-2}
{\rm diam}_{F_1\cup D_0}(P_1\pi_{F_1}(a), \Phi_1(b))<2n.
\end{equation}
Note that $\partial\Delta\cap F_1$ is an element of the image of $a$ by $\pi_A^1:\mathcal{C}^0(\partial V_1)\rightarrow \mathcal{P}(\mathcal{AC}^0(F_1))$.
Further, by the minimality of $|D_a\cap D_0|$, the disk $\Delta$ is essential in $V_1^1$.
Let $D_0^1$ and $D_0^2$ be the two components of $D_0\setminus \Delta$,
and let $\Delta'$ be one of the disks properly embedded in $V_1^1$ which are parallel to $D_0^1\cup \Delta$ or $D_0^2\cup \Delta$.
Then we have $\partial\Delta'\in P_1(\mathcal{D}'(V_1^1))$, and also $\partial\Delta'\in P_1\pi_0(\partial\Delta\cap F_1)\subset P_1\pi_0\pi_A^1(a)=P_1\pi_{F_1}(a)$.
These, together with the inequality (\ref{eqn-2}), imply
\begin{eqnarray*}
d_{F_1\cup D_0}(P_1(\mathcal{D}'(V_1^1)), \Phi_1(\mathcal{D}_2))&\leq& d_{F_1\cup D_0}(\partial\Delta', \Phi_1(b))\\
&\leq&{\rm diam}_{F_1\cup D_0}(P_1\pi_{F_1}(a), \Phi_1(b))\\
&<&2n,
\end{eqnarray*}
which contradicts the inequality (\ref{eqn-1-2}).

\vspace{2mm}
{\it Case 2}. $|D_a\cap D_0|= 0$.
\vspace{2mm}

In this case, the arguments in Case 1 work with regarding $D_a=\Delta'$ to have a contradiction.

The above contradictions give $d_{\partial V_2}(fg(\mathcal{D}(V_1)),\mathcal{D}_2)=n$. 
\end{proof}

\begin{remark}
{\rm
If we pose the assumption that the distance $d(V_1\cup_f V_2)$ of the genus-$g$ Heegaard splitting $V_1\cup_f V_2$ is greater than or equal to $2$ in Proposition \ref{prop-appendix}, then the statement of the proposition can be strengthened as in the following form:
}

Let $D_0$ be a separating essential disk in $V_1$, and let $\mathcal{D}_2$ be any subset of $\mathcal{D}(V_2)$.
If $d_{\partial V_2}(f(\partial D_0),\mathcal{D}_2)=n$,
then there exists a homeomorphism $g:\partial V_1\rightarrow \partial V_1$ 
such that $d_{\partial V_2}(fg(\mathcal{D}(V_1)),\mathcal{D}_2)=n$.\\
{\rm
In fact, the statement can be proved basically by using the arguments of the proof of Proposition \ref{prop-appendix}.
The difference is the proof of inequality (\ref{eqn-0}). 
We should replace it with:

Note that $f(\partial D_0)(=f(\partial F_1)=f(\partial F_2))$ intersects with every essential loop in $\mathcal{D}(V_2)$, since $d_{\partial V_2}(f(\partial D_0),\mathcal{D}(V_2))\geq d(V_1\cup_f V_2)\geq 2$.
By \cite[Theorem 1]{Li}, either 
\begin{equation}\label{eqn-00}
{\rm diam}_{F_i}(\pi_{F_i}f^{-1}(\mathcal{D}_2))\leq {\rm diam}_{F_i}(\pi_{F_i}f^{-1}(\mathcal{D}(V_2)))\leq 12
\end{equation}
or $V_2$ is a $[0,1]$-bundle over $f(F_1)$.
In the latter case, it is easy to see that $g$ must be even and that the union of $V_2$ and $N(D_0)$ is homeomorphic to a $[0,1]$-bundle over a closed surface, say $S$, of genus $g/2$.
Note that the exterior of the union of $V_2$ and $N(D_0)$ is ${\rm Cl}(V_1\setminus N(D_0))$ and consists of two handlebodies of genus $g/2$. 
Thus, $S$ is a Heegaard surface of genus $g/2$, and $\partial V_2(=f(\partial V_1))$ is a stabilization of $S$.
This implies $d(V_1\cup_f V_2)=0$, a contradiction.
Hence, we have the inequality (\ref{eqn-00}).
}
\end{remark}

\begin{proof}[Proof of Corollary \ref{thm-2}]
We first note that the proof of the corollary for the case when $n=2$ is exceptional, 
and we give it in Appendix of this paper, and 
in this proof we show the corollary for the case $n \geq 3$. 
Let $C_1\cup_P C_2=C_1\cup_f C_2$ be a genus-$g$ Heegaard splitting with distance $n(\geq 3)$ obtained in Theorem \ref{thm-1}.
By the proof of Theorem \ref{thm-1}, there are separating essential disks $D_1$ and $D_2$ in $C_1$ and $C_2$, respectively, such that $d_{\partial_+C_2}(f(\partial D_1),\partial D_2)=n$.
Let $H_i$ $(i=1,2)$ be a handlebody of genus $(g-1)$.
Take and fix any homeomorphism $h_i:\partial H_i\rightarrow \partial_-C_i$, and put $V_i:=C_i\cup_{h_i}H_i$(, hence, $V_i$ is a handlebody of genus $g$).
Then $V_1\cup_f V_2$ is a genus-$g$ Heegaard splitting.

By Proposition \ref{prop-appendix}, there exists a homeomorphism $g_1:\partial V_1\rightarrow\partial V_1$ such that $d_{\partial V_2}(fg_1(\mathcal{D}(V_1)), \partial D_2)=n$.
By applying Proposition \ref{prop-appendix} again to $V_2\cup_{(fg_1)^{-1}}V_1$, we see that there exists a homeomorphism $g_2:\partial V_2\rightarrow\partial V_2$ such that $$d_{\partial V_1}((fg_1)^{-1}g_2(\mathcal{D}(V_2)), \mathcal{D}(V_1))=n.$$
That is, the distance of the Heegaard splitting $V_1\cup_{g_2^{-1}fg_1}V_2$ is exactly $n$.
\end{proof}

\appendix
\section*{Appendix (A construction of distance 2 examples)}

In this Appendix, we show for each $g \ge 2$, there is a genus-$g$
Heegaard splitting of a closed 3-manifold with distance 2. 
The examples are given by using the construction of strongly irreducible 
Heegaard splittings in \cite{K-R}. 
For the description of the construction we will use the notations 
$(H, A_1 \cup A_2)$, $N$, $R$ etc. in Section~2.1 of \cite{K-R}. 

For the case when $g=2$, let $F$ be an annulus, and 
let $R = F \times [0,1]$. 
For the case when $g\geq 3$, 
let $F$ be a genus-($g-2$) non-orientable surface (: connected sum of 
$g-2$ copies of projective planes)
with two holes, and  
let $R$ be the orientable twisted $[0,1]$-bundle over $F$. 
Note that $F$ is homotopy equivalent to a bouquet of  $g-1$ circles, hence 
$R$ is homeomorphic to the genus-($g-1$) handlebody. 
Let $R'$ be a copy of $R$. 
Then let $\mathcal{A}_1 \cup \mathcal{A}_2$ 
(resp. $\mathcal{A}_1' \cup \mathcal{A}_2'$) be the union of annuli in
$\partial R$ (resp. $\partial R'$) corresponding to the $[0,1]$-bundle over $\partial F$. 
Then let $N$ be the manifold obtained from 
$R$ and $R'$ by identifying the subsurfaces of the boundaries 
corresponding to the associated $\partial [0,1]$-bundle. 
It is easy to see that the manifolds $N$, $R$, $R'$ satisfy 
the conditions (1), (2), (3) in the page 639 of \cite{K-R}.

Recall from \cite{K-R} that 
$H$ is a genus-2 handlebody, and $\{ A_1,  A_2\}$ is a pair of primitive annuli in $\partial H$. 
Let 
$(H', A_1' \cup A_2')$ 
be a copy of 
$(H, A_1 \cup A_2)$. 
Then it is observed in \cite{K-R} that for any 2-bridge link $L$ in $S^3$ 
there is a homeomorphism 
$h: 
\text{Cl}(\partial H \setminus (A_1 \cup A_2)) 
\rightarrow 
\text{Cl}(\partial H' \setminus (A_1' \cup A_2'))$
such that 
the manifold obtained from $H$ and $H'$ by identifying 
$\text{Cl}(\partial H \setminus (A_1 \cup A_2)) $ and 
$\text{Cl}(\partial H' \setminus (A_1' \cup A_2'))$ by $h$ 
is homeomorphic to the exterior $E(L)$ of $L$. 
Then let $M$ be the 3-manifold obtained from $E(L)$ and $N$ 
by identifying their boundaries by an orientation-reversing homeomorphism 
such that $\mathcal{A}_i$ (resp. $\mathcal{A}_i'$) is identified with $A_i$ (resp. $A_i'$). 
Then it is shown in Section~2.1 of \cite{K-R} that $H \cup R$ and $H' \cup R'$ are genus-$g$ 
handlebodies, and these handlebodies give a Heegaard splitting of $M$. 

Then we have: 

\noindent 
{\bf Assertion.}
{\it  
Suppose that the 2-bridge link $L$ is not a trivial link or a Hopf link, then 
the distance of the Heegaard splitting $(H \cup R) \cup (H' \cup R')$ is exactly 2. 
}

\begin{proof}
Since $L$ is not a trivial link or a Hopf link, we see, by Proposition~2.1 of \cite{K-R}, 
that $(H \cup R) \cup (H' \cup R')$ is strongly irreducible, i.e. the distance of the Heegaard splitting 
is greater than or equal to 2.
On the other hand, since $\partial E(L)$ ($=\partial N$) $\subset M$ is an essential 
torus, we see, by \cite{Har}, that the distance of any Heegaard splitting of $M$ is at most $2$, 
and this together with the above shows that the distance of the Heegaard splitting 
$(H \cup R) \cup (H' \cup R')$ is exactly 2. 

\end{proof}

\end{document}